\documentclass[11pt,leqno]{article}
\usepackage{amsmath, amscd, amsthm, amssymb, graphics, xypic, mathrsfs, setspace, fancyhdr, times, bm, enumitem}
\usepackage[colorlinks=true,pagebackref=true]{hyperref} 
\hypersetup{backref}

\newcommand{\tensor}{\otimes}
\newcommand{\colim}{\operatorname{colim}}

\newcommand{\Spec}{\operatorname{Spec}}

\newcommand{\isomt}{{\stackrel{{\scriptscriptstyle{\sim}}}{\;\rightarrow\;}}}

\newcommand{\sma}{{\scriptstyle{\wedge}}}

\renewcommand{\O}{{\mathcal O}}
\renewcommand{\hom}{\operatorname{Hom}}

\newcommand{\cplx}{{\mathbb C}}

\newcommand{\Z}{{\mathbb Z}}

\newcommand{\aone}{{\mathbb A}^1}
\newcommand{\pone}{{\mathbb P}^1}

\newcommand{\ho}[1]{\mathscr{H}({#1})}

\newcommand{\DM}{{\mathbf{DM}}}

\newcommand{\bpi}{\bm{\pi}}

\newcommand{\et}{\mathrm{\acute et}}

\newcommand{\SH}{{\mathbf{SH}}}

\newcommand{\Sm}{\mathrm{Sm}}

\newcommand{\Spc}{\mathrm{Spc}}

\newcommand{\K}{{{\mathbf K}}}

\renewcommand{\H}{{{\mathbf H}}}

\newcommand{\Addresses}{{
  \bigskip
  \footnotesize

  A.~Asok, \textsc{Department of Mathematics, University of Southern California, 3620 S Vermont Ave KAP 104,
    Los Angeles, CA 90089-2532, United States;} \textit{E-mail address:} \url{asok@usc.edu}

}}

\newcounter{intro}
\theoremstyle{plain}
\newtheorem{thm}{Theorem}[subsection]

\newtheorem{lem}[thm]{Lemma}
\newtheorem{cor}[thm]{Corollary}
\newtheorem{prop}[thm]{Proposition}
\newtheorem*{claim*}{Claim}  

\newtheorem*{thm*}{Theorem}
\newtheorem*{problem*}{Problem}

\newtheorem{thmintro}{Theorem}
\newtheorem{propintro}[thmintro]{Proposition}

\theoremstyle{definition}
\newtheorem{defn}[thm]{Definition}

\theoremstyle{remark}
\newtheorem{rem}[thm]{Remark}

\newtheorem{ex}[thm]{Example}

\numberwithin{equation}{section}

\begin{document}
\pagestyle{fancy}
\renewcommand{\sectionmark}[1]{\markright{\thesection\ #1}}
\fancyhead{}
\fancyhead[LO,R]{\bfseries\footnotesize\thepage}
\fancyhead[LE]{\bfseries\footnotesize\rightmark}
\fancyhead[RO]{\bfseries\footnotesize\rightmark}
\chead[]{}
\cfoot[]{}
\setlength{\headheight}{1cm}

\author{Aravind Asok\thanks{Aravind Asok was partially supported by National Science Foundation Award DMS-1254892.}}

\title{{\bf Stably $\aone$-connected varieties \\ and universal triviality of $CH_0$}}
\date{}
\maketitle

\begin{abstract}
We study the relationship between several notions of connectedness arising in $\aone$-homotopy theory of smooth schemes over a field $k$: $\aone$-connectedness, stable $\aone$-connectedness and motivic connectedness, and we discuss the relationship between these notations and rationality properties of algebraic varieties.  Motivically connected smooth proper $k$-varieties are precisely those with universally trivial $CH_0$.  We show that stable $\aone$-connectedness coincides with motivic connectedness, under suitable hypotheses on $k$.  Then, we observe that there exist stably $\aone$-connected smooth proper varieties over the field of complex numbers that are not $\aone$-connected.
\end{abstract}

\begin{flushright}
\begin{small}
{\em To Chuck Weibel on the occasion of his $65^{th}$ birthday.}
\end{small}
\end{flushright}


\section{Introduction}
Suppose $k$ is a field, and $X$ is a connected smooth projective $k$-variety of dimension $d$.  Recall that $X$ is called {\em stably $k$-rational} if there exists an integer $m \geq 0$ such that $k(X)(t_1,\ldots,t_m)$ is a purely transcendental extension of $k$.  Determining whether a ``nearly rational" (e.g., rationally connected) variety is (not) stably $k$-rational is a rather subtle problem in general.  

It has been know for some time that if $X$ is a stably $k$-rational variety, then the Chow group of $0$-cycles is {\em universally trivial}: for every extension $L/k$, the degree map $CH_0(X_L) \to \Z$ is an isomorphism (see, e.g., \cite{Merkurjev}).  The condition of universal triviality of $CH_0$ for a smooth projective variety $X$ is equivalent to the existence of a Chow-theoretic decomposition of the diagonal (see, e.g., \cite{VoisinDiagDecomp} for general context or \cite[Proposition 1.4]{CTP} for this statement).  Recently, Voisin has shown, using degeneration techniques, that studying universal triviality of $CH_0$ is fruitful for disproving (stable) rationality of smooth proper varieties \cite{VoisinSolids}.  Her method was refined and extended by a number of authors \cite{CTP, Totaro} and we now know much more about, e.g., the failure of stable rationality of hypersurfaces in projective spaces.

This paper, inspired by some relationships between rationality questions and $\aone$-homotopy theory explored in \cite{AM}, is centered around the following question: are there invariants of ``homotopic nature" intermediate between universal triviality of $CH_0$ and stable rationality?  To explain the source of this question we observe that, using duality in Voevodsky's derived category of motives in $\DM_k$ (see, e.g., \cite{MVW}), universal triviality of $CH_0$ of $X$ is equivalent to the motive ${\mathbf M}(X)$ having trivial zeroth homology sheaf (see Lemma \ref{lem:ch0universal} for a proof).  Thus, universal triviality of $CH_0$ for a smooth projective variety $X$ is equivalent to the condition of being {\em motivically connected} (see Definition \ref{defn:connectedness}).  On the other hand, in \cite[Theorem 2.3.6]{AM}, it was observed that stably $k$-rational varieties, over fields $k$ having characteristic zero, are necessarily connected from the standpoint of the Morel-Voevodsky unstable $\aone$-homotopy category \cite{MV}.

If we write $\ho{k}$ for the Morel-Voevodsky $\aone$-homotopy category, sending a smooth scheme $X$ to its motive ${\mathbf M}(X)$ factors through a Hurewicz style functor $\Sm_k \to \ho{k} \to \DM_k$. However, as in topology, the Hurewicz homomorphism factors through a stabilization with respect to suspension:
\[
\Sm_k \longrightarrow \ho{k} \longrightarrow \SH_k \longrightarrow \DM_k;
\]
here $\SH_k$ is the stable $\aone$-homotopy category of $\pone$-spectra (see, e.g., \cite{Jardine}).

It makes sense to ask whether the object in $\SH_k$ determined by a smooth scheme $k$ (given by a $\pone$-suspension spectrum) is ``connected" and we will call a smooth projective variety $X$ {\em stably $\aone$-connected} if it is ``connected" in this sense (again, see Definition \ref{defn:connectedness}).  One has the following sequence of implications relating these various notions of connectedness for smooth projective varieties over a field $k$ having characteristic $0$ (see Lemma \ref{lem:implications} and the discussion that precedes it for more precise statements):\newline
\begin{center}
\{stably rational\} $\Longrightarrow$ \{$\aone$-connected\} $\Longrightarrow$ \{stably $\aone$-connected\} $\Longrightarrow$ \{motivically connected\}.\newline
\end{center}
We can refine the vague question posed above to the following: are these implications strict?   The main goal of this note is to answer this question.

\begin{thmintro}[See Theorem~\ref{thm:main}]
\label{thmintro:main}
If $k$ is an infinite perfect field having finite $2$-cohomological dimension, then a smooth proper $k$-variety $X$ is motivically connected if and only if it is stably $\aone$-connected.
\end{thmintro}

The key tools in the proof are (i) the description of the zeroth stable $\aone$-homotopy sheaf of a smooth proper $k$-variety obtained in \cite{AH} in terms of the ``Chow-Witt group of zero cycles" and (ii) a spectral sequence relating a certain filtration on Chow-Witt groups from their description in terms of cohomology of Milnor-Witt K-theory sheaves and the filtration by powers of the fundamental ideal in Milnor-Witt K-theory; this spectral sequence, which is a slight modification of one considered by Pardon and Totaro, was described in \cite{AsokFaselSecondary}.  We refer the reader to Remark \ref{rem:slice} for an ``explanation" of the result involving Voevodsky's slice tower.

\begin{propintro}[See Proposition~\ref{prop:examples}]
\label{propintro:examples}
There exist stably $\aone$-connected smooth proper complex varieties of any dimension $\geq 2$ that fail to be $\aone$-connected.
\end{propintro}

Granted Theorem \ref{thmintro:main}, one knows that there exist stably $\aone$-connected smooth proper varieties over some non-algebraically closed fields (examples arising in the work of Parimala are discussed in Section \ref{ss:mainexample}).  Proposition \ref{propintro:examples} refines this observation by constructing minimal dimensional examples over $\cplx$.  It seems reasonable to expect that examples of stably $\aone$-connected varieties that are not $\aone$-connected exist over algebraically closed fields having positive characteristic as well.

Over fields $k$ having characteristic $0$ there exist examples of $\aone$-connected smooth proper $k$-varieties that are not stably $k$-rational (we explain the examples, which can be found in work of Colliot-Th\'el\`ene--Sansuc on rationality problems for tori,  in Example \ref{ex:nonstablyrationalvarieties}).  We do not know of such an example over an algebraically closed field.


\subsubsection*{Acknowledgements}
The results below were established in the course of the Fall 2015 edition of the Los Angeles algebraic geometry seminar, which focused on stable rationality.  In particular, the author would like to thank Burt Totaro for helpful discussions and comments on an early version of this paper and Sasha Merkurjev for pointing out Example \ref{ex:nonstablyrationalvarieties}.

\subsubsection*{Notation/Conventions}
Throughout this paper the letter $k$ will be used to denote a field.  By a smooth $k$-scheme, we will mean a scheme that is separated, smooth and has finite type over $\Spec k$.  We write $\Sm_k$ for the category of such objects.  By a sheaf on $\Sm_k$, we will always mean a sheaf for the Nisnevich topology.

\section{$\aone$-homotopic notions of connectedness}
Section \ref{ss:homotopycategories} recalls key points of homotopy theory of algebraic varieties in order to make definitions regarding connectivity, i.e., Definition \ref{defn:connectedness}.  Here, we also explain the precise link between ``motivic connectedness" and universal trivality of $CH_0$.  Section \ref{ss:maintheorem} contains the proof of the result that motivically connected smooth proper varieties are stably $\aone$-connected; this result appears as Theorem \ref{thm:main}.  Section \ref{ss:mainexample} contains an analysis of stably $\aone$-connected varieties that are not $\aone$-connected, which appears in Proposition \ref{prop:examples}.

\subsection{Homotopy categories in brief}
\label{ss:homotopycategories}
Write $\Spc_k$ for the category of simplicial presheaves on $\Sm_k$.  Sending a smooth scheme $X$ to the corresponding representable presheaf on $\Sm_k$ and viewing presheaves on $\Sm_k$ as simplicially constant simplicial presheaves, the Yoneda lemma shows that $\Sm_k \to \Spc_k$ is a fully-faithful embedding.  There is an $\aone$-local model structure on $\Spc_k$ \cite[\S 3]{MV}, and we write $\ho{k}$ for the corresponding homotopy category.  We write $\bpi_0^{\aone}(X)$ for the Nisnevich sheaf associated with the presheaf on $\Sm_k$ given by $U \mapsto \hom_{\ho{k}}(U,X)$ and $\ast$ for the representable presheaf corresponding to $\Spec k$.

We can consider the category of symmetric spectra valued in $\Spc_k$; this has a stable model structure and we write $\SH_k$ for the corresponding homotopy category.  A smooth scheme $X$ determines a suspension symmetric $\pone$-spectrum, denoted $\Sigma^{\infty}_{\pone} X_+$, and sending $X \mapsto \Sigma^{\infty}_{\pone} X_+$ extends to a functor $\ho{k} \longrightarrow \SH_k$.  We set $\mathbf{S}^0_k:= \Sigma^{\infty}_{\pone} \Spec k_+$; this is the unit object for a symmetric monoidal structure $\sma$ on $\SH_k$.  As usual, one may discuss notions of connectivity for spectra, and it is a theorem of F. Morel that, at least if $k$ is an infinite field, then $\Sigma^{\infty}_{\pone} X_+$ is a $(-1)$-connected spectrum (see \cite[Theorem 6.1.8]{MStable} and \cite[5.2.1-5.2.2]{MIntro}).  We define the sheaf of stable $\aone$-connected components, denoted $\bpi_0^{s\aone}(X)$, as the Nisnevich sheaf associated with the presheaf $U \mapsto \hom_{\SH_k}(\mathbf{S}^0_k \wedge \Sigma^{\infty}_{\pone} U_+,\Sigma^{\infty}_{\pone} X_+)$.  We set $\bpi_0^{s\aone}(\mathbf{S}^0_k) := \bpi_0^{s\aone}(\Spec k)$.   A celebrated result of Morel \cite[Theorem 6.43]{MField} shows that $\bpi_0^{s\aone}(\Spec k) = \mathbf{GW}$, the unramified Grothendieck-Witt sheaf (see also \cite[\S 6]{Mpi0}).

Voevodsky's category $\DM_k$ is discussed in, e.g., \cite{MVW}.  The functor sending a smooth scheme $X$ to its corresponding representable presheaf with transfers extends to a functor $\SH_k \to \DM_k$.\footnote{Strictly speaking, this is a lie, and we should consider an analog of $\DM_k$ constructed with unbounded complexes.  However, all the objects we consider will be highly connected so no harm is done.}  We write $\mathbf{M}(X)$--the motive of $X$--for the object in $\DM_k$ corresponding to a smooth scheme $X$.  The object $\mathbf{M}(X)$ has an explicit model as a bounded below complex of presheaves with transfers.  By definition, the homology sheaves of this complex (viewed simply as sheaves of abelian groups) are concentrated in degrees $> -1$.  The zeroth homology sheaf $H_0(\mathbf{M}(X))$ is, by definition, the zeroth Suslin homology sheaf $\H_0^S(X)$.

With this notation and terminology, we may introduce more precisely the various notions of connectedness we consider.

\begin{defn}
\label{defn:connectedness}
Suppose $X$ is a smooth $k$-scheme.  We say that $X$ is
\begin{enumerate}[noitemsep,topsep=1pt]
\item {\em $\aone$-connected} if $\bpi_0^{\aone}(X) \cong \ast$;
\item {\em stably $\aone$-connected} if $\bpi_0^{s\aone}(X) \to \bpi_0^{s\aone}(\mathbf{S}^0_k)$ is an isomorphism; or
\item {\em motivically connected} if the canonical map $\H_0^S(X) \to \Z$ is an isomorphism.
\end{enumerate}
\end{defn}

It was observed in \cite[Theorem 2.3.6]{AM} that if $k$ has characteristic $0$, then stably $k$-rational smooth proper $k$-varieties are $\aone$-connected.  Note that the conclusion of \cite[Theorem 2.3.6]{AM} is also true if ``weak factorization" holds over $k$.  Thus, for example, one concludes that smooth proper $k$-rational surfaces over an arbitrary field are $\aone$-connected.  With this in mind, the relationship between the notions in Definition \ref{defn:connectedness} is summarized in the following result.

\begin{lem}
\label{lem:implications}
Suppose $X$ is a smooth proper $k$-variety over an infinite field $k$.
\begin{enumerate}[noitemsep,topsep=1pt]
\item If $X$ is $\aone$-connected, then $X$ is stably $\aone$-connected.
\item If $X$ is stably $\aone$-connected, then $X$ is motivically connected.
\end{enumerate}
\end{lem}

\begin{proof}
These results are essentially consequences of Morel's stable $\aone$-connectivity theorem (it is appeal to this result that forces us to impose the hypothesis that $k$ be infinite).  Indeed, using this result, one equips $\SH_k$ and $\DM_k$ with $t$-structures whose heart can be described explicitly: the heart of $\SH_k$ is the category of ``homotopy modules" \cite[Definition 5.2.4 and Theorem 5.2.6]{MIntro}, while the heart of $\DM_k$ is the category of strictly $\aone$-invariant sheaves with transfers.  The sheaf $\bpi_0^{s\aone}(X)$ is initial among homotopy modules admitting a map from $X$.  With this in mind, Point (1) follows from an evident modification of \cite[Proposition 3.5]{ABirational}.

For Point (2), assume $X$ is stably $\aone$-connected, so $\bpi_0^{s\aone}(X) \cong \mathbf{GW}$.  Now, by assumption, for any object $\mathbf{M}$ in the homotopy $t$-structure on $\SH_k$, the canonical map $\mathbf{M}(k) \to \mathbf{M}(X)$ is a bijection.  In particular, this is true for $\mathbf{M}$ arising from strictly $\aone$-invariant sheaves with transfers.  In that case, any map $\bpi_0^{s\aone}(X) \to \mathbf{M}$ factors uniquely through $\H_0^S(X)$ and we conclude by appealing to \cite[Lemma 3.3]{ABirational} and the Yoneda lemma.
\end{proof}

The following easy (and well-known) lemma provides the connection between connectivity in the sense just defined and universal triviality of $CH_0$ in the sense mentioned in the introduction; we include it here simply for completeness.

\begin{lem}
\label{lem:ch0universal}
If $X$ is a smooth proper $k$-scheme, then $X$ is motivically connected if and only if for every extension field $L/k$, $CH_0(X_L)$ is trivial.
\end{lem}

\begin{proof}
The sheaf $\H_0^S(X)$ is strictly $\aone$-invariant, i.e., all of its cohomology presheaves are $\aone$-invariant.  As a consequence of this, one deduces that the sheaf $\H_0^S(X)$ is determined by its sections on finitely generated extensions of the base field.

If $X$ is an irreducible smooth proper $k$-scheme of dimension $d$, then one deduces that the following sequence of isomorphism hold using duality and the cancellation theorem \cite{VCancellation}:
\[
\begin{split}
\H_0^S(X)(\Spec L) &:= \hom_{\DM_L}(\Z,\mathbf{M}(X_L)) \\
&\cong \hom_{\DM_L}(\mathbf{M}(X_L)(-d)[-2d],\Z) \\
&\cong \hom_{\DM_L}(\mathbf{M}(X_L),\Z(d)[2d]) \\
&\cong CH^{d}(X_L) = CH_0(X_L).
\end{split}
\]
Now, if $\H_0^S(X)(\Spec L)$ is trivial, then $CH_0(X_L) \to \Z$ is an isomorphism.  Since we can write any extension as an increasing union of its finitely generated sub-extensions, we conclude that $CH_0(X_L) \to \Z$ is an isomorphism for extensions $L/k$ that are not necessarily finitely generated.  Conversely, if $CH_0(X_L) \to \Z$ is an isomorphism for all finitely generated extensions, then by the observation of the first paragraph, we conclude that $\H_0^S(X) \to \Z$ is an isomorphism of sheaves.
\end{proof}

\subsection{Stable $\aone$-connectedness vs. motivic connectedness}
\label{ss:maintheorem}
\begin{thm}
\label{thm:main}
Suppose $k$ is an infinite perfect field that has finite $2$-cohomological dimension.  If $X$ is a motivically connected smooth projective $k$-variety, then $X$ is stably $\aone$-connected.
\end{thm}

\begin{proof}
Assume that $\H_0^S(X) \to \Z$ is an isomorphism.  We want to show that $\H_0^{s\aone}(X)\to \mathbf{GW}$ is an isomorphism.  Generalizing the first part of the proof of Lemma \ref{lem:ch0universal}, it was shown in \cite[Theorem 5.3.1]{AH} that $\pi_0^{s\aone}(X)(L) \isomt \widetilde{CH}^n(X_L,\omega_{X_L}) = H^n(X_L,\K^{MW}_n(\omega_{X_L}))$ in a fashion compatible with field extensions and functorial with respect to pushforwards.\footnote{The statement of \cite[Theorem 5.3.1]{AH} also imposes the additional hypothesis that $k$ has characteristic unequal to $2$.  This hypothesis was imposed there to appeal to results about Milnor-Witt cohomology; the necessary results were established in \cite{MField} in the case $k$ has characteristic $2$ as well, at least if $k$ is perfect.}  Here $\widetilde{CH}^n(X_L,\omega_{X_L})$ is the twisted Chow-Witt group introduced in \cite[\S 10]{FaselChowWitt}.

F. Morel showed that there is a short exact sequence of sheaves on the small Nisnevich (or Zariski) site of $X_L$ of the form
\[
0 \longrightarrow \mathbf{I}^{n+1}(\omega_{X_L}) \longrightarrow \K^{MW}_n(\omega_{X_L}) \longrightarrow \K^M_n \longrightarrow 0.
\]
Moreover, $\mathbf{I}^{n+1}(\omega_{X_L})$ admits a decreasing filtration $\mathbf{I}^{n+j}(\omega_{X_L}) \supset \mathbf{I}^{n+j+1}(\omega_{X_L})$ with successive subquotients identified, via the Milnor conjecture on quadratic forms, with $\K^M_{n+2}/2$; see \cite[\S 2.1]{AsokFaselSecondary} for all of this notation.\footnote{The careful reader may want to assume that $k$ has characteristic unequal to $2$.  There is a small problem in the fiber product presentation of the Milnor-Witt K-theory sheaf described in \cite{MPuissances} that is fixed in \cite{GSZ}.  As explained in \cite[p. 54]{MField}, these results can be extended to characteristic $2$ as well, replacing reference to \cite{MPuissances} by \cite{GSZ}.}  We now analyze the pushfoward along $X_L \to \Spec L$ in more detail.

We begin by analyzing the induced push-forward at the level of the graded pieces of the filtration.  The assumption $\H_0^S(X) \to \Z$ is an isomorphism implies that $CH_0(X_L) \cong H^n(X_L,\K^M_n) = \Z$.  The edge map in the hypercohomology spectral sequence for the motivic complexes $\Z(n)$ yields a map $H^{2n+1,,n+1}(X,\Z/2) \to H^{n}(X_L,\K^M_{n+j}/2)$; this map is an isomorphism by connectivity estimates \cite[Lemma 4.11]{VMilnor}.  Again, by duality
\[
H^{2n+j,,n+j}(X,\Z/2) = \hom_{\DM_k}(\mathbf{M}(X),\Z/2(n+j)[2n+j]) \cong \hom_{\DM}(\Z,\mathbf{M}(X) \tensor \Z/2(j)[j]).
\]
Because $\mathbf{M}(X)$ is $(-1)$-connected, the weak K\"unneth formula allows us to identify the first non-zero homology sheaf of $\mathbf{M}(X) \tensor \Z/2(j)[j]$ in terms of Voevodsky's tensor product of the first non-zero homology sheaves of the constituents; this result follows by transposing the proof of \cite[Proposition 3.3.9]{AWW} to $\DM_k$.  Thus, we conclude $H_0(\mathbf{M}(X)\tensor \Z/2(j)[j]) \cong \H_0^S(X) \tensor^{\aone} \K^M_j/2$.  In particular, evaluating on $L$ we conclude that $H_0(\mathbf{M}(X) \tensor \Z/2(1)[1])(L) \cong \K^M_j/2(L)$; once more observe that this identification is functorial with respect to field extensions. Thus, we conclude that the pushfoward map $H^{n}(X_L,\K^M_{n+j}/2) \isomt \K^M_j/2(L)$ is functorial with respect to field extensions.

Next, one sees that pushforward induces the following morphism of exact sequences
\[
\xymatrix{
H^n(X_L,\mathbf{I}^{n+1}(\omega_{X_L})) \ar[r]\ar[d]& H^n(X_L,\K^{MW}_n(\omega_{X_L})) \ar[r]\ar[d]& H^n(X_L,\K^{M}_n) \ar[r]\ar[d] & 0 \\
\ar[r] \mathbf{I}^1(L) \ar[r] & \mathbf{GW}(L) \ar[r] & \Z \ar[r] & 0.
}
\]
Note that the morphism one term to the left is simply the map $H^{n-1}(X_L,\K^M_n) \to 0$ since the group $H^{-1}(\Spec k,\K^M_0)$ is trivial by construction.  The two right vertical maps are isomorphisms.  Therefore, if the map $H^n(X_L,\mathbf{I}^{n+1}(\omega_{X_L})) \to \mathbf{I}^1(L)$ is an isomorphism, then by the five lemma it follows that the map $H^n(X_L,\K^{MW}_n(\omega_{X_L})) \to GW(L)$ is an isomorphism.

Finally, we analyze the pushforward $H^n(X_L,\mathbf{I}^{n+1}(\omega_{X_L})) \to \mathbf{I}^1(L)$ using the (truncated) Pardon spectral sequence from \cite[\S 2.2]{AsokFaselSecondary}; this is the spectral sequence induced by the filtration of the Gersten-Witt complex by sub-complexes corresponding to higher powers of the fundamental ideal.\footnote{When $k$ has characteristic $2$, one appeals to the Rost-Schmid complex described in \cite[Remark 5.13]{MField} instead of the Gersten-Witt complex.}  We observe only that the $d_r$-differential has bidegree $(1,r-1)$ for $r \geq 2$, and $E^{n,\ast}_2 = H^{n}(X_L,\K^M_{n+\ast}/2)$.  Since $k$ is assumed to have finite $2$-cohomological dimension, the truncated Pardon spectral sequence is bounded.  Note that since $X_L$ has dimension $n$, there are no outgoing differentials from the column $E^{n,\ast}$ so $H^n(X_L,\mathbf{I}^{n+1}(\omega_{X_L}))$ admits a filtration whose associated graded is described in terms of quotients of $H^{n}(X_L,\K^M_{n+\ast}/2)$.

Now, recall that pushfoward map $H^n(X_L,\mathbf{I}^{n+1}(\omega_{X_L})) \to \mathbf{I}^1(L)$ comes from a pushforward defined at the level of Gersten-Witt complexes \cite[Th\'eor\`eme 8.3.4 and Corollaire 10.4.5]{FaselChowWitt}; this complex-level pushforward respects the filtration by powers of the fundamental ideal and thus induces a morphism of truncated Pardon spectral sequences.  Since we already saw that the maps $H^{n}(X_L,\K^M_{n+j}/2) \to K^M_{j}/2(L)$ are isomorphisms for every $j \geq 0$, the morphism of Pardon spectral sequences induced by pushforward thus identifies the $n$-th column of the $E_2$-page of the truncated Pardon spectral sequence for $H^n(X_L,\mathbf{I}^{n+1}(\omega_{X_L}))$ with the $0$-th column of the $E_2$-page of the truncated Pardon spectral sequence for $H^0(\Spec L,\mathbf{I}^1)$.  In particular, we conclude that all incoming differentials on this column are trivial, and the isomorphism we hoped to construct is an immediate consequence.
\end{proof}

\begin{rem}
\label{rem:slice}
The aforementioned result should not be viewed as surprising.  Indeed, the proof given above can be viewed as an explicit form of an argument involving Voevodsky's slice filtration \cite{VoevodskySlice}.  Indeed, the zero slice of the sphere spectrum is Voevodsky's motivic Eilenberg-MacLane spectrum $\mathbf{H}\Z$ and each slice is a module over $\mathbf{H}\Z$ \cite{LevineSlice}.  If $k$ has finite cohomological dimension, the slice filtration of $\Sigma^{\infty}_{\pone} X_+$ yields a convergent spectral sequence allowing us to compute $\pi_0^{s\aone}(X)$ in terms of information from motivic cohomology \cite{LevineSliceConvergence}.  That the filtration by powers of the fundamental ideal coincides with the slice filtration is established in \cite{LevineGWSlice}.  We leave it to the interested reader to recast the result in those terms.  Note, however, that our assumption that $k$ has finite $2$-cohomological dimension uses slightly weaker hypotheses than the argument just sketched.
\end{rem}

\begin{rem}
It would be interesting to know whether Theorem \ref{thm:main} holds without the assumption that $k$ has finite $2$-cohomological dimension.
\end{rem}

In line with Remark \ref{rem:slice}, Theorem \ref{thm:main} admits an interpretation in the theory of birational motives developed in \cite{KahnSujathaI,KahnSujathaII}.  More precisely, combining \cite[Proposition 3.1.1]{KahnSujathaI}, Lemma \ref{lem:ch0universal} and Theorem \ref{thm:main}, we deduce the following result (see also \cite[Theorem 2.1]{TotaroClassifying} for other equivalent characterizations).

\begin{cor}
\label{cor:aoneconnectedtrivialbirationalmotive}
Assume $k$ is a field having finite $2$-cohomological dimension.  If $X$ is a smooth proper $k$-variety, then $X$ is stably $\aone$-connected if and only if the birational motive of $X$ is trivial.
\end{cor}

\subsection{$\aone$-connectedness vs. stable $\aone$-connectedness}
\label{ss:mainexample}
If $k$ is not algebraically closed, then it is known that one can find smooth proper $k$-varieties that are motivically connected but not $\aone$-connected since there are varieties that have a $0$-cycle of degree $1$ but no $k$-rational point (see \cite[Example 4.19]{ABirational} where an example due to Parimala is discussed).  We now improve this observation by producing minimal dimensional counterexamples over an algebraically closed field.   The moral of the story is that the idea of ``becoming connected after $\pone$-suspension is rather subtle".  We begin with the following lemma, which goes back to Bloch-Srinivas \cite[Remark 2 p. 1252]{BlochSrinivas}.

\begin{lem}
\label{lem:universaltrivialityofch0surfaces}
If $X$ is a smooth complex surface such that $CH_0(X) \cong \Z$ and such that $H^*(X,\Z)$ is torsion free, then $\H_0^S(X) = \Z$.
\end{lem}

\begin{proof}
In light of Lemma \ref{lem:ch0universal}, this is an immediate consequence of \cite[Corollary 2.2]{VoisinCubic} or \cite[Proposition 1.9]{ACTP}.
\end{proof}

Recall that if $X \in \Sm_k$, then two points $x_0,x_1 \in X(k)$ are said to be elementarily $R$-equivalent if, setting $k[t]_{(0,1)}$ to be the semi-localization of $k[x]$ at $0$ and $1$, there exists an element of $x(t) \in X(k[t]_{(0,1)})$ such that $x(0) = x_0$ and $x(1) = x_1$.  Following Manin, we say that two elements of $X(k)$ are {\em $R$-equivalent} if they are equivalent for the equivalence relation generated by elementary $R$-equivalence; we write $X(k)/R$ for the set of $R$-equivalence classes of $k$-rational points.  If $X,Y \in \Sm_k$, then $(X \times Y)(k)/R \cong X(k)/R \times Y(k)/R$.

\begin{prop}
\label{prop:examples}
Suppose $k$ is a perfect field.
\begin{enumerate}[noitemsep,topsep=1pt]
\item Every stably $\aone$-connected smooth proper curve is $\aone$-connected.
\item For every integer $d \geq 2$, there exist stably $\aone$-connected smooth projective $\cplx$-varieties of dimension $d$ that are not $\aone$-connected.
\end{enumerate}
\end{prop}

\begin{proof}
For Point (1) suppose $C$ is a smooth projective curve.  By \cite[Theorem 3.1]{SuslinVoevodsky}, we know that $\H_0^S(C) = \mathbf{Pic}(C)$.  In particular, if $C$ has genus $> 0$ then $\mathbf{Pic}(C)(\bar{k}) = Pic(C_{\bar{k}})$, which is not isomorphic to $\Z$.  Thus, we can assume without loss of generality that $C$ has genus $0$, i.e., it is a conic.  In that case, since $\H_0^S(C) \cong \Z$, we conclude that $C$ contains a $0$-cycle of degree $1$, i.e., it has rational points of extensions of coprime degree.  In particular, it has a rational point after making an extension of odd degree.  Now any conic that has a rational point over an extension of odd degree is necessarily isomorphic to $\pone$ (this follows, e.g., from a theorem of Springer \cite[Theorem VII.2.7]{Lam}, but is much easier).  The fact that $\pone$ is $\aone$-connected can be seen in many ways; we leave this to the reader.

By Lemma \ref{lem:universaltrivialityofch0surfaces}, if $Y$ is any smooth proper surface with torsion free integral cohomology and $CH_0(S) \cong \Z$, then $Y$ is motivically connected.  Now, let $S$ be any smooth proper surface as in the previous sentence with Kodaira dimension $> -\infty$; the class of such surfaces is non-empty as, after the work of Voisin \cite[Corollary 2.5]{VoisinBC}, for example, it contains Barlow surfaces \cite{Barlow}.  By Theorem \ref{thm:main}, any such $S$ is stably $\aone$-connected as well.

On the other hand, a smooth proper variety $S$ is $\aone$-connected if and only if it is universally $R$-trivial by \cite[Theorem 2.4.3]{AM} (i.e., for every finitely generated separable extension $L/k$, the set $X(L)/R = \ast$).  In particular, $\aone$-connected varieties are rationally connected (more weakly, they have Kodaira dimension $-\infty$).  However, the $S$ from the previous paragraph has Kodaira dimension $> 0$ by assumption.

Finally, we produce examples of dimension $> 2$.  Since $S$ is not $\aone$-connected, there exists a finitely generated separable extension $L/k$ such that $S(L)/R \neq \ast$.  Then, for any $\aone$-connected variety $X$, since the map $S \times X \to S$ induces a bijection $(S \times X)(L)/R \to S(L)/R$, we conclude that $S \times X$ is not $\aone$-connected.  However, one again applying the weak Kunneth formula (as in the proof of Theorem \ref{thm:main}) we conclude that $\H_0^{S}(S \times X) \cong \H_0^S(S)$ as well.
\end{proof}

\begin{ex}
\label{ex:nonstablyrationalvarieties}
If $k$ is not algebraically closed, then there exist examples of $\aone$-connected smooth projective varieties that are not stably $k$-rational.  Suppose $T$ is a torus that is retract $k$-rational but not stably $k$-rational.  By \cite[Proposition 13]{CTSToresI}, any smooth compactification $X$ of $T$ is universally $R$-trivial.  Therefore by \cite[Corollary 2.4.4]{AM} we conclude that $X$ is $\aone$-connected and, by assumption, not stably rational.  On the other hand, such tori actually exist; see also \cite[Theorems 1.3-1.7]{HoshiYamasaki} and the references therein.
\end{ex}

\begin{rem}
Write $\Omega_{\pone}$ for the derived $\pone$-loop space functor on $\ho{k}$.  For any pointed space $(\mathscr{X},x)$, there are canonical morphisms $\mathscr{X} \to \Omega_{\pone} \Sigma_{\pone} \mathscr{X} \to \Omega_{\pone}^2 \Sigma_{\pone}^2 \mathscr{X} \to \cdots$.  We take $\mathscr{X} = X_+$ for a smooth proper $k$-scheme $X$.  To say that $X$ is stably $\aone$-connected is equivalent to saying that the map $\colim_n \Omega_{\pone}^n \Sigma_{\pone}^n X_+ \to \colim_n \Omega_{\pone}^n \Sigma_{\pone}^n \Spec k_+$ is an isomorphism after applying $\bpi_0^{\aone}(-)$.  If $X = \Spec k$, then the value of this colimit is actually achieved for $n = 2$ \cite[Theorem 6.43]{MField}.  Thus, given a stably $\aone$-connected smooth proper $k$-variety, it makes sense to ask whether there exists a finite integer $n$ such that $\Omega_{\pone}^n \Sigma_{\pone}^n X_+ \to \Omega_{\pone}^n \Sigma_{\pone}^n \Spec k_+$ becomes an isomorphism after applying $\bpi_0^{\aone}(-)$ for all $N \geq n$ (cf. \cite[Remark 2]{AHRequivalence} for closely related discussion).
\end{rem}

\begin{rem}
As a counter-point to Proposition \ref{prop:examples}, we recall that a smooth proper $k$-variety $X$ is $\aone$-connected if and only if $X$ is connected, $X(k)$ is non-empty, and the generic point $\eta \in X(k(X))$ is $R$-equivalent to $x$.  This result follows immediately by combining \cite[Theorem 2.4.3]{AM} and \cite[Theorem 8.5.1]{KahnSujathaLocalization}; the proof is a straightforward consequence of existence of specialization maps for $R$-equivalence classes (see, e.g., D. Madore \cite[Proposition 3.1]{Madore}, J. Koll\'ar \cite{KollarSpecialization}) or \cite[Lemma 6.2.3]{AM}).   We view this characterization of $\aone$-connectedness as analogous to the characterization of motivic connectedness in terms of a Chow-theoretic decomposition of the diagonal.
\end{rem}

\subsection{Birational invariance of some higher Chow groups}
\label{ss:complements}
Even though Theorem \ref{thm:main} shows that $\pi_0^{s\aone}(X)$ is no more refined of an invariant than $\H_0^S(X)$, it nevertheless highlights other witnesses to failure of universal triviality of $CH_0$.  Indeed, the functor $L \mapsto H^{2n+j,n+j}(X_L,\Z/2)$, which appears in the proof of Theorem \ref{thm:main} extends to a strictly $\aone$-invariant sheaf (with transfers) on $\Sm_k$ (this extension can be defined as the sheafification for the Nisnevich topology of $U \mapsto H^{2n+j,n+j}(U \times X,\Z/2)$); we will write $\mathbf{H}^{2n+j,n+j}(X,\Z/2)$ for this sheaf.  We observe, that these sheaves provide additional stable birational invariants of $X$.

One knows that motivic cohomology groups $H^{p,q}(X,\Z)$ of a smooth variety $X$ vanish for $p > q + \dim X$.  The line of groups $H^{p,q}$ with $p = q + \dim X$ need not vanish and appear in the proof of Theorem \ref{thm:main}.  We now observe that these groups are, quite generally, stable birational invariants of smooth proper varieties; we rephrase this in homological terms to eliminate dependence on dimension.

\begin{prop}
For any integer $j \geq 0$ and any commutative ring $R$, the motivic Borel-Moore homology group $H_{-j,-j}(X,R) := \hom_{\DM}(R(-j)[-j],\mathbf{M}(X))$ is a stable birational invariant of smooth proper $k$-varieties.
\end{prop}

\begin{proof}
We use various standard properties of motivic cohomology; see \cite[\S 16]{MVW} for details.  By definition: $H_{-j,-j}(X \times {\mathbb P}^n) = \hom_{\DM_k}(R(-j)[-j]),\mathbf{M}(X \times {\mathbb P}^n)$.  By the projective bundle formula $\hom_{\DM_k}(R(-j)[-j]),\mathbf{M}(X \times {\mathbb P}^n) \cong \hom_{\DM_k}(R(-j)[-j],\oplus_{i=0}^n \mathbf{M}(X)(i)[2i])$.  Again by Voevodsy's cancellation theorem \cite{VCancellation}, we know that $\hom_{\DM_k}(R(-j)[-j],\mathbf{M}(X)(i)[2i]) \cong \hom_{\DM_k}(R,\mathbf{M}(X)(i+j)[2i+j])$ and by duality, it is easy to see that this group vanishes if $i > 0$.

To conclude, it suffices to observe that $H_{-j,-j}(X,R)$ is a birational invariant, and this follows by an evident modification of the usual proof by the action of correspondences (\cite[Lemma 16.1.11]{Fulton}).
\end{proof}

We can sheafify these groups for the Nisnevich topology: simply consider $\mathbf{H}_{-j,-j}(X) := H_0(\mathbf{M}(X)(j)[j])$; these groups can also be defined as the sheafification for the Nisnevich topology of the presheaf $U \mapsto H^{2\dim X + j,\dim X + j}(U \times X,\Z)$.

\begin{prop}
Suppose $X$ is a smooth proper variety.  If $\H_0^S(X) \to \Z$, then for every $j \geq 0$, the map $\mathbf{H}_{-j,-j}(X) \to \mathbf{H}_{-j,-j}(\Spec k) = \K^M_j$ is an isomorphism.
\end{prop}

\begin{proof}
Since $\mathbf{H}_{-j,-j}(X) := H_0(\mathbf{M}(X)(j)[j])$, it follows from the weak K\"unneth theorem (see the proof of Theorem \ref{thm:main}; one adapts the proof in \cite[Proposition 3.3.9]{AWW}) that $H_0(\mathbf{M}(X)(j)[j]) \cong \H_0^S(X) \tensor^{\aone} \H_0(\Z(j)[j]) \cong \H_0^S(X) \tensor^{\aone} \K^M_j$.
\end{proof}

\begin{rem}
Assume one is working over an algebraically closed base field $k$.  There is a cycle class map $H^{2d + j,d+j}(X,\Z)$ is $H^{2d+j}_{\et}(X,\Z\ell(d+j))$.  If $j > 0$, this map is zero for dimensional reasons.  Therefore, if $\mathbf{H}_{-j,-j}(X)$ is a non-trivial invariant, it detects arithmetic as opposed to topological information.
\end{rem}

We conclude with the following example, which shows that the groups $H_{-1,-1}(X)$ actually do naturally arise.

\begin{ex}
The groups $\H_{-1,-1}(X)$ can also be identified as $H^d(X,\K^Q_{d+1}) = H^{d}(X,\K^M_{d+1})$ for a smooth variety of dimension $d$ (here $\K^Q_{i}$ is the sheafification of the Quillen K-theory presheaf on $\Sm_k$ for the Nisnevich topology).  In this guise, they appear in the very definition of norm varieties and thus the proof of the Bloch-Kato conjecture \cite[Theorem 0.1]{SuslinJoukhovitski}.
\end{ex}

\begin{footnotesize}
\bibliographystyle{alpha}
\bibliography{stableA1connectedness}

\begin{thebibliography}{ACTP13}

\bibitem[ACTP13]{ACTP}
A.~Auel, J.-L. Colliot-Th{\'e}l{\`e}ne, and R.~Parimala.
\newblock Universal unramified cohomology of cubic fourfolds containing a
  plane.
\newblock 2013.
\newblock {\em Preprint}, available at \url{http://arxiv.org/abs/1310.6705},.

\bibitem[AF15]{AsokFaselSecondary}
A.~Asok and J.~Fasel.
\newblock Secondary characteristic classes and the {E}uler class.
\newblock {\em Doc. Math.}, (Extra volume: Alexander S. Merkujev's sixtieth
  birthday):7--29, 2015.

\bibitem[AH11a]{AH}
A.~Asok and C.~Haesemeyer.
\newblock The $0$-th stable {$\mathbb{A}^1$}-homotopy sheaf and quadratic zero
  cycles.
\newblock 2011.
\newblock {\em Preprint}, available at \url{http://arxiv.org/abs/1108.3854}.

\bibitem[AH11b]{AHRequivalence}
A.~Asok and C.~Haesemeyer.
\newblock Stable {$\Bbb A^1$}-homotopy and {$R$}-equivalence.
\newblock {\em J. Pure Appl. Algebra}, 215(10):2469--2472, 2011.

\bibitem[AM11]{AM}
A.~Asok and F.~Morel.
\newblock Smooth varieties up to {$\Bbb A^1$}-homotopy and algebraic
  {$h$}-cobordisms.
\newblock {\em Adv. Math.}, 227(5):1990--2058, 2011.

\bibitem[Aso13]{ABirational}
A.~Asok.
\newblock Birational invariants and {$\Bbb A^1$}-connectedness.
\newblock {\em J. Reine Angew. Math.}, 681:39--64, 2013.

\bibitem[AWW15]{AWW}
A.~Asok, K.~Wickelgren, and T.B. Williams.
\newblock The simplicial suspension sequence in {$\mathbb{A}^1$}-homotopy.
\newblock 2015.
\newblock {\em Preprint}, available at \url{http://arxiv.org/abs/1507.05152}.

\bibitem[Bar85]{Barlow}
R.~Barlow.
\newblock A simply connected surface of general type with {$p_g=0$}.
\newblock {\em Invent. Math.}, 79(2):293--301, 1985.

\bibitem[BS83]{BlochSrinivas}
S.~Bloch and V.~Srinivas.
\newblock Remarks on correspondences and algebraic cycles.
\newblock {\em Amer. J. Math.}, 105(5):1235--1253, 1983.

\bibitem[CTP14]{CTP}
J.-L. Colliot-Th{\'e}l{\`e}ne and A.~Piruka.
\newblock Hypersurfaces quartiques de dimension $3$: non rationalit{\'e}
  stable.
\newblock {\em {\em To appear} Ann. Sci. \'Ecole Norm. Sup.}, 2014.
\newblock {\em Preprint}, available at \url{http://arxiv.org/abs/1402.4153}.

\bibitem[CTS77]{CTSToresI}
J.-L. Colliot-Th{\'e}l{\`e}ne and J.-J. Sansuc.
\newblock La {$R$}-\'equivalence sur les tores.
\newblock {\em Ann. Sci. \'Ecole Norm. Sup. (4)}, 10(2):175--229, 1977.

\bibitem[Fas08]{FaselChowWitt}
J.~Fasel.
\newblock Groupes de {C}how-{W}itt.
\newblock {\em M\'em. Soc. Math. Fr. (N.S.)}, (113):viii+197, 2008.

\bibitem[Ful98]{Fulton}
W.~Fulton.
\newblock {\em Intersection theory}, volume~2 of {\em Ergebnisse der Mathematik
  und ihrer Grenzgebiete. 3. Folge. A Series of Modern Surveys in Mathematics}.
\newblock Springer-Verlag, Berlin, second edition, 1998.

\bibitem[GSZ16]{GSZ}
S.~Gille, S.~Scully, and C.~Zhong.
\newblock Milnor--{W}itt {$K$}-groups of local rings.
\newblock {\em Adv. Math.}, 286:729--753, 2016.

\bibitem[HY15]{HoshiYamasaki}
A.~Hoshi and A.~Yamasaki.
\newblock Rationality problem for algebraic tori.
\newblock {\em {\em To appear} Mem. Amer. Math. Soc.}, 2015.
\newblock {\em Preprint} available at \url{http://arxiv.org/abs/1210.4525}.

\bibitem[Jar00]{Jardine}
J.~F. Jardine.
\newblock Motivic symmetric spectra.
\newblock {\em Doc. Math.}, 5:445--553 (electronic), 2000.

\bibitem[Kol04]{KollarSpecialization}
J.~Koll{\'a}r.
\newblock Specialization of zero cycles.
\newblock {\em Publ. Res. Inst. Math. Sci.}, 40(3):689--708, 2004.

\bibitem[KS15a]{KahnSujathaLocalization}
B.~Kahn and R.~Sujatha.
\newblock Birational {G}eometry and {L}ocalization of {C}ategories.
\newblock {\em Doc. Math.}, (Extra volume: Alexander S. Merkujev's sixtieth
  birthday):277--334, 2015.

\bibitem[KS15b]{KahnSujathaI}
B.~Kahn and R.~Sujatha.
\newblock Birational motives {I}: pure birational motives.
\newblock {\em {\em To appear}, Ann. K-Theory}, 2015.
\newblock {\em Preprint}, available at \url{http://arxiv.org/abs/0902.4902}.

\bibitem[KS15c]{KahnSujathaII}
B.~Kahn and R.~Sujatha.
\newblock Birational motives {II}: {T}riangulated birational motives.
\newblock 2015.
\newblock {\em Preprint}, available at \url{http://arxiv.org/abs/1506.08385}.

\bibitem[Lam05]{Lam}
T.~Y. Lam.
\newblock {\em Introduction to quadratic forms over fields}, volume~67 of {\em
  Graduate Studies in Mathematics}.
\newblock American Mathematical Society, Providence, RI, 2005.

\bibitem[Lev08]{LevineSlice}
M.~Levine.
\newblock The homotopy coniveau tower.
\newblock {\em J. Topol.}, 1(1):217--267, 2008.

\bibitem[Lev11]{LevineGWSlice}
M.~Levine.
\newblock The slice filtration and {G}rothendieck-{W}itt groups.
\newblock {\em Pure Appl. Math. Q.}, 7(4, Special Issue: In memory of Eckart
  Viehweg):1543--1584, 2011.

\bibitem[Lev13]{LevineSliceConvergence}
M.~Levine.
\newblock Convergence of {V}oevodsky's slice tower.
\newblock {\em Doc. Math.}, 18:907--941, 2013.

\bibitem[Mad05]{Madore}
D.~Madore.
\newblock Sur la sp{\'e}cialisation de la {R}-{\'e}quivalence.
\newblock 2005.
\newblock {\em Preprint}, available at
  \url{http://perso.telecom-paristech.fr/~madore/specialz.pdf},.

\bibitem[Mer08]{Merkurjev}
A.~Merkurjev.
\newblock Unramified elements in cycle modules.
\newblock {\em J. Lond. Math. Soc. (2)}, 78(1):51--64, 2008.

\bibitem[Mor04a]{MIntro}
F.~Morel.
\newblock An introduction to {$\mathbb A^1$}-homotopy theory.
\newblock In {\em Contemporary developments in algebraic {$K$}-theory}, ICTP
  Lect. Notes, XV, pages 357--441 (electronic). Abdus Salam Int. Cent. Theoret.
  Phys., Trieste, 2004.

\bibitem[Mor04b]{Mpi0}
F.~Morel.
\newblock On the motivic {$\pi_0$} of the sphere spectrum.
\newblock In {\em Axiomatic, enriched and motivic homotopy theory}, volume 131
  of {\em NATO Sci. Ser. II Math. Phys. Chem.}, pages 219--260. Kluwer Acad.
  Publ., Dordrecht, 2004.

\bibitem[Mor04c]{MPuissances}
F.~Morel.
\newblock Sur les puissances de l'id\'eal fondamental de l'anneau de {W}itt.
\newblock {\em Comment. Math. Helv.}, 79(4):689--703, 2004.

\bibitem[Mor05]{MStable}
F.~Morel.
\newblock The stable {${\mathbb A}^1$}-connectivity theorems.
\newblock {\em $K$-Theory}, 35(1-2):1--68, 2005.

\bibitem[Mor12]{MField}
F.~Morel.
\newblock {\em {$\mathbb A^1$}-algebraic topology over a field}, volume 2052 of
  {\em Lecture Notes in Mathematics}.
\newblock Springer, Heidelberg, 2012.

\bibitem[MV99]{MV}
F.~Morel and V.~Voevodsky.
\newblock {${\mathbf A}^1$}-homotopy theory of schemes.
\newblock {\em Inst. Hautes \'Etudes Sci. Publ. Math.}, (90):45--143 (2001),
  1999.

\bibitem[MVW06]{MVW}
C.~Mazza, V.~Voevodsky, and C.~Weibel.
\newblock {\em Lecture notes on motivic cohomology}, volume~2 of {\em Clay
  Mathematics Monographs}.
\newblock American Mathematical Society, Providence, RI; Clay Mathematics
  Institute, Cambridge, MA, 2006.

\bibitem[SJ06]{SuslinJoukhovitski}
A.~Suslin and S.~Joukhovitski.
\newblock Norm varieties.
\newblock {\em J. Pure Appl. Algebra}, 206(1-2):245--276, 2006.

\bibitem[SV96]{SuslinVoevodsky}
A.~Suslin and V.~Voevodsky.
\newblock Singular homology of abstract algebraic varieties.
\newblock {\em Invent. Math.}, 123(1):61--94, 1996.

\bibitem[Tot14]{TotaroClassifying}
B.~Totaro.
\newblock The motive of a classifying space.
\newblock 2014.
\newblock {\em Preprint}, available at \url{http://arxiv.org/abs/1407.1366}.

\bibitem[Tot15]{Totaro}
B.~Totaro.
\newblock Hypersurfaces that are not stably rational.
\newblock {\em {\em To appear} J. Amer. Math. Soc.}, 2015.
\newblock {\em Preprint} available at \url{http://arxiv.org/abs/1502.04040}.

\bibitem[Voe02]{VoevodskySlice}
V.~Voevodsky.
\newblock A possible new approach to the motivic spectral sequence for
  algebraic {$K$}-theory.
\newblock In {\em Recent progress in homotopy theory ({B}altimore, {MD},
  2000)}, volume 293 of {\em Contemp. Math.}, pages 371--379. Amer. Math. Soc.,
  Providence, RI, 2002.

\bibitem[Voe03]{VMilnor}
V.~Voevodsky.
\newblock Motivic cohomology with {${\mathbf Z}/2$}-coefficients.
\newblock {\em Publ. Math. Inst. Hautes \'Etudes Sci.}, (98):59--104, 2003.

\bibitem[Voe10]{VCancellation}
V.~Voevodsky.
\newblock Cancellation theorem.
\newblock {\em Doc. Math.}, (Extra volume: Andrei A. Suslin sixtieth
  birthday):671--685, 2010.

\bibitem[Voi14a]{VoisinBC}
C.~Voisin.
\newblock Bloch's conjecture for {C}atanese and {B}arlow surfaces.
\newblock {\em J. Differential Geom.}, 97(1):149--175, 2014.

\bibitem[Voi14b]{VoisinDiagDecomp}
C.~Voisin.
\newblock {\em Chow rings, decomposition of the diagonal, and the topology of
  families}, volume 187 of {\em Annals of Mathematics Studies}.
\newblock Princeton University Press, Princeton, NJ, 2014.

\bibitem[Voi15a]{VoisinCubic}
C.~Voisin.
\newblock On the universal {$CH_0$}-group of cubic hypersurfaces.
\newblock {\em {\em To appear} Jour. Eur. Math. Soc.}, 2015.
\newblock {\em Preprint}, available at \url{http://arxiv.org/abs/1407.7261},.

\bibitem[Voi15b]{VoisinSolids}
C.~Voisin.
\newblock Unirational threefolds with no universal codimension {$2$} cycle.
\newblock {\em Invent. Math.}, 201(1):207--237, 2015.

\end{thebibliography}
\end{footnotesize}
\Addresses
\end{document}